\documentclass[12pt]{amsart}

\usepackage{tikz}

\usepackage[utf8]{inputenc}

\usepackage{amsmath}
\usepackage{amssymb}
\usepackage{amsfonts}
\usepackage{amsthm}
\usepackage{graphicx}
\usepackage{enumerate}
\usepackage{a4wide}
\usepackage{tabularx}
\usepackage{subcaption}

\makeatletter
\@namedef{subjclassname@2010}{
  \textup{2010} Mathematics Subject Classification}
\makeatother

\usepackage{booktabs}

\usepackage{hyperref}

\newtheorem{theorem}{Theorem}

\newtheorem{corollary}[theorem]{Corollary}
\newtheorem{lemma}[theorem]{Lemma}
\newtheorem{proposition}[theorem]{Proposition}
\newtheorem{algorithm}{Algorithm}

\newtheorem{definition}[theorem]{Definition}
\newtheorem{question}[theorem]{Question}
\newtheorem{example}[theorem]{Example}
\numberwithin{equation}{section}

\def\R{\mathbb{R}}
\def\C{\mathbb{C}}
\def\N{\mathbb{N}}
\def\Z{\mathbb{Z}}
\def\Q{\mathbb{Q}}
\def\al{\alpha}

\def\ll{\mathcal{L}}
\def\nn{\mathcal{N}}
\def\bb{\mathcal{B}}
\def\ff{\mathcal{F}}
\def\mm{\mathcal{M}}
\def\cc{\mathcal{C}}
\def\dd{\mathcal{D}}

\def\uu{\mathcal{U}}

\def\Rem{\mathcal{R}}
\def\G{\mathcal{G}}
\def\V{\mathcal{V}}

\def\E{\mathcal{E}}

\newcommand{\abs}[1]{\left|#1\right|}
\newcommand{\sm}[2]{#1\!\setminus\!#2}

\usepackage{color}

\begin{document}

\title[On Littlewood and Newman multiples of Borwein polynomials]{On Littlewood and Newman multiples of Borwein polynomials}
\author{P. Drungilas, J. Jankauskas, J. \v Siurys}
\address{Department of Mathematics and Informatics, Vilnius
University, Naugarduko 24, Vilnius LT-03225, Lithuania}
\email{pdrungilas@gmail.com}
\address{Mathematik und Statistik, Montanuniversit\"at Leoben, Franz Josef Stra\ss{}e 18, A-8700 Leoben, Austria}
\email{jonas.jankauskas@gmail.com}
\address{Department of Mathematics and Informatics, Vilnius
University, Naugarduko 24, Vilnius LT-03225, Lithuania}
\email{jonas.siurys@mif.vu.lt}

\thanks{The first author is supported by  the Research Council of Lithuania grant MIP-049/2014. The second author is supported by project P27050 \emph{Fractals and Words: Topological, Dynamical, and Combinatorial Aspects} funded by the Austrian Science Fund (FWF)}

\subjclass[2010]{11R09, 11Y16, 12D05} \keywords{Borwein polynomial, Littlewood polynomial, Newman polynomial, Pisot number, Salem Number, 
Mahler measure, polynomials of small height}

\begin{abstract}
A Newman polynomial has all the coefficients in $\{ 0,1\}$ and constant term 1, whereas a Littlewood polynomial 
has all coefficients in $\{-1,1\}$. We call $P(X)\in\Z[X]$ a \emph{Borwein} polynomial if all its coefficients belong to $\{ -1,0,1\}$ and $P(0)\neq 0$. 
By exploiting an algorithm which decides whether a given monic integer polynomial with no roots on the unit circle $|z|=1$ has a non-zero multiple in $\Z[X]$ with coefficients in a finite set $\dd \subset \Z$,  for every 
Borwein polynomial of degree at most 9 we determine whether it divides any Littlewood or Newman polynomial. In particular, we show that every Borwein polynomial of degree at most 8 which divides some Newman polynomial  divides some Littlewood polynomial as well. In addition to this, for every Newman polynomial of degree at most 11, we check whether it has a Littlewood multiple, extending the previous results of Borwein, Hare, Mossinghoff, Dubickas and Jankauskas.
\end{abstract}

\maketitle

\section{Introduction}

Let $d \in \N$ and let $P(X)$ be a polynomial
\begin{equation}\label{px}
P(X) = a_d X^d + a_{d-1}X^{d-1}+\dots+a_1X+a_0
\end{equation}
in one variable $X$ with integer coefficients $a_j \in \Z$. To avoid trivialities,  we consider only polynomials with non-zero leading and constant terms $a_d \cdot a_0 \ne 0$. In such case, both $P(X)$ and it's reciprocal polynomial $P^{*}(X) := X^dP(1/X)$ are of the same degree $d$. If $P(X)$ has only three non-zero coefficients $a_j$, for $0 \leq j \leq d$, then it is called \emph{a trinomial}. Similarly, if the number of non-zero coefficients is four, $P(X)$ is called \emph{a quadrinomial}.

We call the polynomial $P(X)$ in \eqref{px} a \emph{Littlewood} polynomial, if $a_j \in \{-1, 1\}$ for each $0 \leq j \leq d$. For instance,  $P(X) = X^4+X^3-X^2+X-1$  is a Littlewood polynomial. The set of all Littlewood polynomials is denoted by $\ll$.

Similarly, a polynomial $P(X)$ is called \emph{a Newman polynomial}, if all coefficients $a_j \in \{0, 1\}$ and $P(0)=1$. For instance, $P(X) = X^3+X+1$ is a Newman polynomial. The subset of $\Z[X]$ of all Newman polynomials is denoted by $\nn$.

Finally,  an integer polynomial $P(X)$ in \eqref{px} with all coefficients $a_j \in \{-1, 0, 1\}$ and a nonzero constant term 
$P(0)$ is called \emph{a Borwein polynomial}\footnote{This notation in honor of P. Borwein for his work on polynomials of this type was proposed by C. Smyth during the 2015 workshop \emph{The Geometry, Algebra and Analysis of Algebraic numbers} in Banff, Alberta (personal communication).}.  $P(X) = X^5-X^2+1$  is an example of a Borwein polynomial. The set of all Borwein polynomials is denoted by $\bb$.  One has trivial set relations $\nn \subset \bb$, $\ll \subset \bb$. 

We say that a polynomial $P(X)$ has a Littlewood multiple, if it divides some polynomial in the set $\ll$. In the similar way, we say that $P(X)$ has a Newman multiple, or a Borwein multiple, if $P(X)$ divides some polynomial in $\nn$  or in $\bb$, respectively.

When we need to restrict our attention only to polynomials of fixed degree, we use the subscript $d$ in $\nn_d$, $\ll_d$ and $\bb_d$ to denote the sets of Newman, Littlewood and Borwein polynomials of degree $d$, respectively. Similarly, we use subscript ``$\leq d$'' to indicate the sets of polynomials of degree \emph{at most} $d$, that is
\[
\nn_{\leq d} = \bigcup_{j=0}^d \nn_j, \qquad \ll_{\leq d} = \bigcup_{j=0}^d \ll_j, \qquad \bb_{\leq d} = \bigcup_{j=0}^d \bb_j.
\]

Clearly, non-constant polynomials $P(X)$ with all non-negative coefficients 
 cannot have any positive real zeros $X \in [0, \infty)$. Newman 
  polynomials are among such polynomials. To denote the subsets of 
   Littlewood or Borwein polynomials with no real positive zeros, we append 
    the $"-``$ superscript, for instance,  $\ll^-$, $\bb^-$, $\ll_d^-$,  $\bb_d^-$ and $\ll_{\leq d}^-$,  $\bb_{\leq d}^-$.  

Let $\mathcal{A} \subset \Z[X]$. We will employ the notation $\ll(\mathcal{A})$ to denote the set of 
polynomials $P(X) \in \mathcal{A}$ which divide some Littlewood polynomial. Similarly, denote by 
$\nn(\mathcal{A})$ the set of polynomials $P(X) \in \mathcal{A}$ which divide some Newman polynomial. 
In particular, the set $\bb_d\!\setminus\!\ll(\bb)$ consists of those Borwein polynomials of degree $d$ 
that do not divide any Littlewood polynomial, whereas the set $\nn(\bb_d)\!\setminus\!\ll(\bb)$ consists of those  
Borwein polynomials of degree $d$ that divide some Newman polynomial and do not divide any Littlewood polynomial.

Let $\dd\subset\Z$ be a finite set. We call $\dd$ a \emph{digit set}. Central to our work is a further development  (see Section~\ref{sal}) of an algorithm that can answer the following question.

\begin{question}\label{p136}
Given a monic polynomial $P \in \Z[X]$ which has no roots on the unit circle $|z|=1$ in the complex plane, does there 
exist a nonzero polynomial with coefficients in $\dd$ which is divisible by $P$?
\end{question}

The first instance of such an algorithm that we are aware of appeared in the work  of Lau \cite{ksl}. It was specialized to the case when $P(X)$ is a minimal polynomial of a Pisot number. Subsequent computations were done by Borwein and Hare \cite{BH}, Hare and Mossinghoff \cite{HM}. It was used for the computations of the discrete spectra of Pisot numbers. In a special case where the set $\dd = \{-q, \dots, -1, 0, 1, \dots, q\}$, here $q$ is a positive integer, the fact the $P(X)$ has a non-zero multiple $Q(X)$ with coefficients in $\dd$ is equivalent to the fact that the number $0$ has a non-trivial representation in the difference set of the spectra generated by the root $\al$ of $P(X)$ with digits $\{0, 1, \dots, q\}$. Stankov \cite{dst} extended the algorithm to non--Pisot algebraic integers with no conjugates on the unit circle. Akiyama, Thuswaldner and Zaimi \cite{ATZ} show that there exists a finite automaton that can determine the minimal height polynomial with integer coefficients for a given algebraic number provided it has no algebraic conjugates on the unit circle $|z|=1$ in the complex plane.

Thus the previously existing version of this algorithm answers Question~\ref{p136} for irreducible monic polynomials $P(X) \in \Z[X]$ with no roots on the unit circle. One  contribution of our paper is a further development of this algorithm to allow $P(X)$ to have repeated roots (i.e. when $P(X)$ is not separable). This should open the way to answer some problems regarding the multiplicity of the divisors of polynomials with restricted coefficients (see, e.g., Example~\ref{exmr} in Section~\ref{calc}). The condition that $P(X)$ has no roots with $|z|=1$, cannot be dropped, as it is essential to the proof that search terminates, but in some cases these restrictions can be circumvented (see Subsection \ref{nocycl} on cyclotomic factors and the last note at the end of Section \ref{sal}).

We must mention other approaches to search for Newman and Littlewood multiples of $P(X)$ that appear in the literature: Borwein and Hare \cite{BH} made applications of the LLL algorithm to this problem, Mossinghoff \cite{Mos} considered the factorization of Littlewood polynomials of large degrees, Dubickas and Jankauskas \cite{DJ} performed the search for the multipliers of bounded height. However, these heuristic approaches do not allow to identify $P(X)$ that have no such multiple.

We implement our algorithm to answer this question for all Borwein polynomials of degree up to 9 and the digit sets 
$\dd = \{0,1\}$ and $\dd=\{-1,1\}$. In other words, for every Borwein polynomial of degree at most 9 we decide whether 
it has a Littlewood multiple and whether it divides some Newman polynomial. Moreover, for every Newman polynomial 
$P(X)$ of degree at most $11$ we determine whether $P(X)\in\ll(\nn)$. These computations allow us to extend the results 
previously obtained by Dubickas and Jankauskas \cite{DJ}, Borwein and Hare \cite{BH}, Hare and Mossinghoff \cite{HM} (see Section~\ref{rbss} and  Section~\ref{calc}).

This paper is organized as follows. The main results are given in Section~\ref{rbss}. In Section~\ref{calc} we describe our computations. The algorithm along with the proofs of auxiliary results are given in Section~\ref{sal}.

\section{Main results}\label{rbss}

\subsection{Relations between sets $\bb$,  $\ll(\bb)$ and $\nn(\bb)$}\label{rbss}

The set $\bb$ of Borwein polynomials can be decomposed into the following four pairwise disjoint subsets (see Figure~\ref{F}):\\[0,1cm]
\begin{tabularx}{\linewidth}{ r X }
 $\ll(\bb)\!\setminus \!\nn(\bb)$ &-- the set of Borwein polynomials that have Littlewood 
 multiples and don't have Newman multiples;\\[0,1cm]
$\nn(\bb)\!\setminus \!\ll(\bb)$ &-- the set of Borwein polynomials that have Newman 
 multiples and don't have Littlewood multiples;\\[0,1cm]
$\ll(\bb)\cap \nn(\bb)$ &-- the set of Borwein polynomials that have Littlewood and 
  Newman multiples;\\[0,1cm]
$\bb\!\setminus\!\Big( \ll(\bb)\cup \nn(\bb)\Big)$ &-- the set of Borwein polynomials that 
 divide no Littlewood and no Newman polynomial.
 \end{tabularx}
 \vspace{0,2cm}
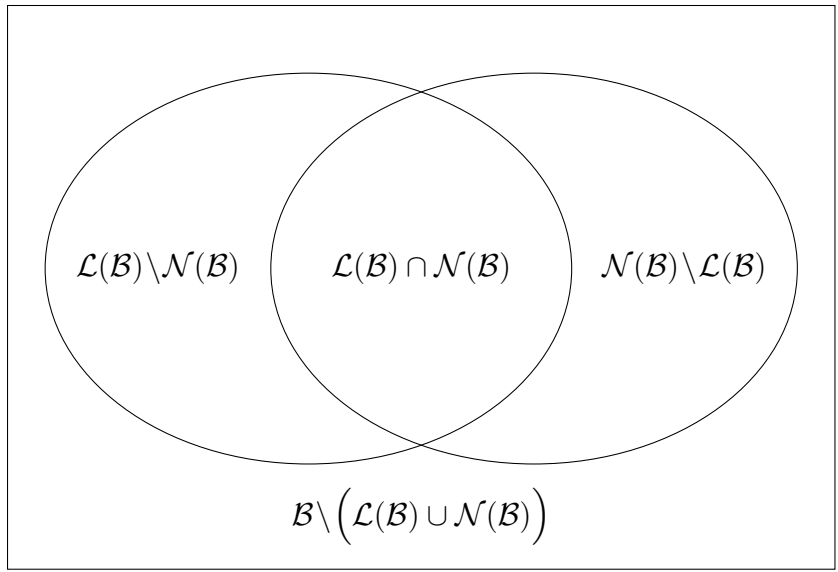
\begin{figure}[h]
\centering
\begin{tikzpicture}
\draw (-5.5,-4) rectangle (5.5,3.5);
\draw (-1.5,0) ellipse (3.5cm and 2.6cm); 
\draw (1.5,0) ellipse (3.5cm and 2.6cm);
\draw (-3.5,0) node {$\ll(\bb)\!\setminus \!\nn(\bb)$};
\draw (3.5,0) node {$\nn(\bb)\!\setminus \!\ll(\bb)$};
\draw (0,0) node {$\ll(\bb)\cap \nn(\bb)$};
\draw (0,-3.3) node {$\bb\!\setminus\!\Big( \ll(\bb)\cup \nn(\bb)\Big)$};
\end{tikzpicture}
\caption {Decomposition of the set of Borwein polynomials.}\label{F}
\end{figure}
 
The rectangle in Figure~\ref{F} corresponds to 
the set $\bb$ while the left and the right hand side ellipses correspond to the sets $\ll(\bb)$ and 
$\nn(\bb)$, respectively.

We implemented Algorithm~\ref{bfs} (see Section~\ref{calc}) and ran it to verify the statements $P(X)\in\ll(\bb)$ and $P(X)\in\nn(\bb)$ 
for all Borwein polynomials $P(X)$ of degree at most 9. Thus we have completed the classification of polynomials from $\bb_{\leq 9}$ 
started by Dubickas and Jankauskas \cite{DJ}. In particular, we calculated the numbers 
\[
\# \big(\ll(\bb_d)\!\setminus\!\nn(\bb)\big), \;\;\# \big(\nn(\bb_d)\!\setminus\!\ll(\bb)\big)\; \text{ and }
\; \# \big(\ll(\bb_d)\cap\nn(\bb_d)\big) 
\]
for every $d\in\{1,2,\dotsc,9\}$ that are provided in Table~\ref{LneNneLirN}. 
As a result we obtain the following statement (see the third column in Table~\ref{LneNneLirN}).
\begin{theorem}\label{B8}
Every Borwein polynomial of degree at most $8$ which divides some Newman polynomial 
divides some Littlewood polynomial as well.
\end{theorem}

Theorem~\ref{B8} is a generalization of Theorem~2 in \cite{DJ} where it is proved that every Newman 
polynomial of degree at most $8$ divides some Littlewood polynomial.

\begin{table}[h]
 \caption {}\label{LneNneLirN}
\begin{tabular}{rccc}
\toprule
 $d$ & $\# \big(\ll(\bb_d)\!\setminus\!\nn(\bb)\big)$ & $\# \big(\nn(\bb_d)\!\setminus\!\ll(\bb)\big)$ & 
 $\# \big(\ll(\bb_d)\cap\nn(\bb_d)\big)$ \\
\midrule
$ 1 $ & $ 2 $ & $ 0 $ & $ 2 $ \\
$ 2 $ & $ 6 $ & $ 0 $ & $ 6 $ \\
$ 3 $ & $ 24 $ & $ 0 $ & $ 12 $ \\
$ 4 $ & $ 72 $ & $ 0 $ & $ 32 $ \\
$ 5 $ & $ 224 $ & $ 0 $ & $ 68 $ \\
$ 6 $ & $ 612 $ & $ 0 $ & $ 164 $ \\
$ 7 $ & $ 1518 $ & $ 0 $ & $ 342 $ \\
$ 8 $ & $ 3610 $ & $ 0 $ & $ 822 $ \\
$ 9 $ & $ 8564 $ & $ 60 $ & $ 1596 $ \\
\bottomrule
 \end{tabular}
 \end{table} 

From Figure~\ref{F}, by the inclusion-exclusion principle, one obtains the following equalities
\begin{align*}
&\# \bb_d\!\setminus\!\ll(\bb) = \#\bb_d - \# \big(\ll(\bb_d)\!\setminus\!\nn(\bb)\big) - 
\# \big(\ll(\bb_d)\cap\nn(\bb_d)\big),\\
&\# \bb_d\!\setminus\!\nn(\bb) = \#\bb_d - \# \big(\nn(\bb_d)\!\setminus\!\ll(\bb)\big) - 
\# \big(\ll(\bb_d)\cap\nn(\bb_d)\big),\\
&\# \bb_d\!\setminus\!(\ll(\bb)\cup\nn(\bb)) = \#\bb_d  - \# \big(\ll(\bb_d)\!\setminus\!\nn(\bb)\big)\\   
& \hspace{4,7cm}-  \# \big(\nn(\bb_d)\!\setminus\!\ll(\bb)\big) - 
\# \big(\ll(\bb_d)\cap\nn(\bb_d)\big),
\end{align*}
which are valid for all positive integers $d$. These numbers, for $d\in\{1,2,\dotsc,9\}$, are given in Table~\ref{DDI}.

\begin{table}[h]
 \caption {}\label{DDI}
\begin{tabular}{rccc}
\toprule
 $d$ & $\# \big(\bb_d\!\setminus\!\ll(\bb)\big)$ & $\# \big(\bb_d\!\setminus\!\nn(\bb)\big)$ & 
 $\# \Big(\bb_d\!\setminus\!(\ll(\bb)\cup\nn(\bb))\Big)$  \\
\midrule
$ 1 $ & $ 0 $ & $ 2 $ & $ 0 $ \\
$ 2 $ & $ 0 $ & $ 6 $ & $ 0 $ \\
$ 3 $ & $ 0 $ & $ 24 $ & $ 0 $ \\
$ 4 $ & $ 4 $ & $ 76 $ & $ 4 $ \\
$ 5 $ & $ 32 $ & $ 256 $ & $ 32 $ \\
$ 6 $ & $ 196 $ & $ 808 $ & $ 196 $ \\
$ 7 $ & $ 1056 $ & $ 2574 $ & $ 1056 $ \\
$ 8 $ & $ 4316 $ & $ 7926 $ & $ 4316 $ \\
$ 9 $ & $ 16084 $ & $ 24588 $ & $ 16024 $ \\
\bottomrule
 \end{tabular}
 \end{table}

For example, there are exactly $196$ Borwein polynomials of degree 6 which have no Littlewood multiple.

\subsection{Borwein polynomials that do not divide any Littlewood polynomial} 

Recall that a real algebraic integer $\al > 1$ is called a \emph{Pisot number} after \cite{Pi}, if all the algebraic conjugates of $\al$ over $\Q$ (other than $\al$ itself) are of modulus $|z|<1$. Similarly, a real algebraic integer $\al > 1$ is called a \emph{Salem number} (see, e.g., \cite{Sa1, Sa2, Sa3}), if all other conjugates of $\al$ lie in the unit circle $|z| \leq 1$ with at least one conjugate on the unit circle $|z|=1$. 

In their computation of the discrete spectra of Pisot numbers, Borwein and Hare \cite{BH} found first 
examples of Borwein polynomials $P(X)$ that divide no Littlewood polynomial. All these polynomials are of 
degree $d=9$ or $d=10$ and they are minimal polynomials of Pisot numbers, see Table \ref{PnotL}. 
So the sets $\bb_9\!\setminus\!\ll(\bb)$ and 
$\bb_{10}\!\setminus\!\ll(\bb)$ are non-empty. 

\begin{table}[h]
 \caption {Minimal polynomials of Pisot numbers 
 that divide no Littlewood polynomial found by Borwein and Hare.}\label{PnotL}
\begin{tabular}{clc}
\toprule
\# & Polynomial $P(X) \in \bb$ & Pisot number\\
\midrule
1 & $X^{10}-X^8-X^7-X^6-X^5+1$ & $1.954062236\ldots$\\
2 & $X^9-X^8-X^7-X^6-X^5-X^4+1$ & $1.963515789\ldots$\\
3 & $X^9-X^8-X^7-X^6-X^5-X^4-X^3-X-1$ & $1.992483962\ldots$\\
4 & $X^9-X^8-X^7-X^6-X^5-X^4-X^3-X^2-1$ & $1.994016415\ldots$\\
\bottomrule
 \end{tabular}
 \end{table} 

In the present work, we find the least degree Borwein polynomials with no Littlewood multiple.

\begin{proposition}
The smallest degree Borwein polynomial which does not divide any Littlewood polynomial is $p(X)=X^4 + X^3 -X + 1$. 
Moreover, 
\[
\bb_{\leq 4}\!\setminus\!\ll(\bb)=\{ \pm p(X), \pm p^*(X)\}.
\]
\end{proposition}
 
A systematic investigation of the sets $\ll(\bb) \cap \nn(\bb)$ and $\nn(\bb) \!\setminus\! \ll(\bb)$ 
was started by Dubickas and Jankauskas in \cite{DJ}. They found that each $P(X) \in \nn_{\leq 8}$ 
 has a Littlewood multiple, so that $\ll(\nn_{\leq 8})=\nn_{\leq 8}$. 
First known polynomials $P(X) \in \nn_9$ that do not divide any polynomial in $\ll$ were 
also identified in \cite{DJ}. They are equal to one of the polynomials no. 1, 3, 5, 9 of Table~\ref{NpndLp} or their reciprocals.  Moreover, all the possible candidates  of $P(X) \in \nn_9$ with no Littlewood multiple were identified (see Table~7 in \cite{DJ}) but not fully resolved. 

\begin{table}[h]
 \caption {The complete set $\nn_{9}\!\setminus\! \ll(\nn)$ (reciprocals omitted).}\label{NpndLp}
\begin{tabular}{cl}
\toprule
\# & Polynomial $P(X)$\\
\midrule
1& $ X^{9} + X^{6} + X^{2} + X + 1 $ \\
2& $ X^{9} + X^{7} + X^{6} + X^{2} + 1 $ \\
3& $ X^{9} + X^{7} + X^{6} + X^{4} + 1 $ \\
4& $ X^{9} + X^{8} + X^{6} + X^{5} + X^{2} + 1 $ \\
5& $ X^{9} + X^{8} + X^{7} + X^{5} + X^{3} + 1 $ \\
6& $ X^{9} + X^{8} + X^{7} + X^{5} + X^{2} + X + 1 $ \\
7& $ X^{9} + X^{8} + X^{5} + X^{3} + X^{2} + X + 1 $ \\
8& $ X^{9} + X^{7} + X^{6} + X^{3} + X^{2} + X + 1 $ \\
9& $ X^{9} + X^{8} + X^{5} + X^{4} + X^{3} + X^{2} + 1 $ \\
\bottomrule
 \end{tabular}
 \end{table}

Our recent computations confirm that none of these candidates divide any Littlewood polynomial. They are listed as polynomials no. 2, 4, 6, 7, 8  (or their reciprocals) in Table~\ref{NpndLp}. Hence, the sets $\ll(\nn_9)$ and $\nn_9\! \setminus\! \ll(\nn)$ are now completely determined. In particular, $\#\nn_{9}\!\setminus\! \ll(\nn) = 18$. The complete list of Newman polynomials of degree $9$ that have no Littlewood multiple is provided in Table~\ref{NpndLp} (with reciprocals omitted) .   

 In the present work, we have been able to extend the classification of the polynomials from the set $\nn_9$ to larger degrees. As these tables are longer and its not practical to provide them here, we just indicate that the sets $\ll(\nn_d)$,  $\nn_d\!\setminus\!\ll(\nn)$ have been completely determined for $d=10$ and $d=11$. In particular, our computations show that $\#\nn_{10}\!\setminus\! \ll(\nn) = 36$ and $\#\nn_{11}\!\setminus\! \ll(\nn) = 174$.

  Using the polynomial $P(X)$ no.3 from Table \ref{NpndLp}, Dubickas and Jankauskas \cite{DJ}
  proved that for all sufficiently large positive integers $n$ the polynomial 
  $x^n P(X) + 1$ does not divide any Littlewood polynomial. This implies that the set 
  $\nn_d\!\setminus\! \ll(\nn)$ is non-empty for all sufficiently large $d$. 
 In addition to this, they proved that every Borwein polynomial with three non-zero terms
 \[
 X^{b} \pm X^{a} \pm 1, \quad 1 \leq a < b, \quad a, b \in\Z
 \] (including Newman trinomials $X^b+X^a+1$)  has a Littlewood multiple, as well as some types of Borwein quadrinomials $X^c\pm X^b\pm X^a\pm1$ do. These results show that set  $\ll(\bb) \cap \nn(\bb)$ has a non-trivial structure. 
 
 Dubickas and Jankauskas \cite{DJ} asked whether there exists a Borwein quadrinomial that does not 
 divide any Littlewood polynomial. Our computations imply that there are exactly 20 such quadrinomials of 
 degree $\leq 9$. 
They are given in Table~\ref{BQndL} (we only list quadrinomials with positive leading coefficient).

\begin{table}[h!]
 \caption{Monic quadrinomials  in $\bb_{\leq 9}\! \setminus\! \ll(\bb)$.}\label{BQndL}
\begin{tabular}{l}
\toprule
$ X^{4} + X^{3} - X + 1 $ \\
$ X^{4} - X^{3} + X + 1 $ \\
$ X^{6} - X^{5} - X - 1 $ \\
$X^{6} + X^{5} + X - 1 $ \\
$ X^{8} - X^{5} + X^{3} + 1 $ \\
$ X^{8} + X^{5} - X^{3} + 1 $ \\
$ X^{8} + X^{7} - X + 1 $ \\
$ X^{8} - X^{7} + X + 1 $ \\
$ X^{8} + X^{6} - X^{2} + 1 $ \\
$ X^{8} - X^{6} + X^{2} + 1 $ \\
\bottomrule
 \end{tabular}
 \end{table}

Since every Borwein trinomial divides some Littlewood polynomial (see \cite[Theorem~1]{DJ}),  
we have the following result (see also Table~\ref{BQndL}).
\begin{corollary}
Number $4$ is the lest positive integer $k$ for which there exists a Borwein polynomial with $k$ nonzero 
terms that divides no Littlewood polynomial.
\end{corollary}

Our computations also show that each quadrinomial in $\nn_{\leq 11}$ divides some 
Littlewood polynomial. Therefore  the following question is of interest.

\begin{question}
Does there exist a Newman quadrinomial with no Littlewood multiple? 
Equivalently, does the set $\nn \!\setminus\! \ll(\nn)$ contain a quadrinomial?
\end{question}

\noindent If such quadrinomial exists, it must be of degree $ \geq 12$.
 
\subsection{Borwein polynomials that do not divide any Newman polynomial} Recall that a Newman
polynomial has no nonnegative real roots. However, not every polynomial $P(X) \in \bb^{-}$ divides a Newman polynomial. Our computations show that 
 
\begin{proposition}
The smallest degree Borwein polynomial without nonnegative real roots and no 
Newman multiple is
\[
p(X)=X^3 + X^2 - X+ 1,
\] 
and
\[
\bb^{-}_{\leq 3}\!\setminus\!\nn(\bb)=\{ \pm p(X), \pm p^*(X)\}.
\]
\end{proposition}

Recall that the \emph{Mahler measure} of a polynomial 
\[
p(X) = \sum_{i=k}^{n}a_kX^k = a_n\prod_{k=1}^{n}(X-\al_k)\in\C[X]
\]
is defined by 
\[
M(p) = |a_n|\prod_{k=1}^{n}\max\{1, |\al_k| \}.
\]
 Hare and Mossinghoff \cite{HM} considered the following problem: does there 
exist a real number $\sigma>1$ such that if $f(X)\in\Z[X]$ has no nonnegative real roots and $M(f)<\sigma$, then 
$f(X)$ divides some Newman polynomial $F(X)$? Based on the results of Dufresnoy and Pisot\cite{DP}, Amara \cite{A} and 
Boyd \cite{B1,B2} they proved that every negative Pisot number which has no positive real algebraic conjugate and is larger than $-\tau$, where $\tau= (1+\sqrt{5})/2\approx 1.61803$ is the golden ratio, is a root of some Newman polynomial. They also proved that certain negative Salem numbers greater that $-\tau$ are roots of Newman polynomials. Moreover, they have constructed a number of polynomials that have Mahler measure 
less than $\tau$, have no positive real roots and yet do not divide any Newman polynomial. 
The smallest Mahler measure in their list is approximately $1.556$ attained by the polynomial $X^6-X^5-X^3+X^2+1$.   
We found that among Borwein polynomials of degree at most 9 there are exactly 16 polynomials which extend this list and 
have Mahler measure less than $1.556$. They are given in Table~\ref{NMlHM} (we omit reciprocal polynomials).

\begin{table}[h]
 \caption {Polynomials in $\bb_{\leq 9}^{-} \!\setminus \!\nn(\bb)$ of small Mahler measure.}\label{NMlHM}
\begin{tabular}{ll}
\toprule
Polynomial $P(X) \in \bb_{\leq 9}^{-} \!\setminus \!\nn(\bb)$ & Mahler measure\\
\midrule
$ X^{9} + X^{8} + X^{7} - X^{5} - X^{4} - X^{3} + 1 $\;\;\;\; & $ 1.436632261 $\\ 
$ X^{9} + X^{8} - X^{3} - X^{2} + 1 $ & $ 1.483444878 $ \\
$ X^{9} - X^{7} - X^{5} + X^{3} + X + 1 $ & $ 1.489581321 $\\ 
$ X^{8} - X^{7} - X^{4} + X^{3} + 1 $ & $ 1.489581321 $\\
$ X^{8} + X^{7} - X^{3} - X^{2} + 1 $ & $ 1.518690904 $\\
$ X^{8} + X^{7} + X^{6} - X^{4} - X^{3} - X^{2} + 1 $ & $ 1.536566472 $\\ 
$ X^{9} - X^{8} - X^{6} + X^{5} + 1 $ & $ 1.536913983 $\\ 
$ X^{9} + X^{5} - X^{3} - X^{2} + 1 $ & $ 1.550687063 $\\  
\bottomrule
 \end{tabular}
 \end{table}
 
 Note that the third polynomial in Table~\ref{NMlHM} factors as $(X + 1) \cdot (X^{8} - X^{7} - X^{4} + X^{3} + 1)$ and 
 the second factor is the fourth polynomial of the table. All the other polynomials in this table (except for the third one) are irreducible over $\Z$.

\subsection{Examples with special factors} 

The following example demonstrates that if two polynomials have Littlewood multiples, their product not necessarily has one.

\begin{example}
The Newman polynomial 
\[
P(X)=X^{11} + X^{10}+X^9+X^8 + X^7 + X^5 + X^4 + X^3  + 1\in\nn
\]
which factors (over $\Z$) as
\[
P(X) = (X^2 + X  + 1)(X^4 + X^3  + 1)(X^5-X^4 + X^3 -X  + 1)
\]
has no Littlewood multiple, although both noncyclotomic factors of $P(X)$ have Littlewood multiples. 
\end{example}

Moreover, $p(X)\in\nn(\bb)$ not necessarily implies $p(X)p^*(X)\in\nn(\bb)$, as can be seen from Example~\ref{psmp}.

\begin{example}\label{psmp}
Let
\[
p(X) = X^3 - X + 1
\]
be the minimal polynomial of the largest negative Pisot number $-\theta \approx -1.32472$.  Both $p(X)$ and its reciprocal $p^{*}(X) = X^3 - X^2 + 1$ have Newman multiples
\[
P(X) = X^5+X^4+1 \quad \text{ and } \quad P^{*}(X) = X^5+X+1,
\]
respectively.

However, the product
\[
p(X)p^{*}(X) = X^6-X^5-X^4+3X^3-X^2-X+1
\]
has no Newman multiple. In contrast, $p(X)p^{*}(X)$ divides Borwein polynomial
\[
\begin{array}{rcl}
	Q(X) &=& (X^2 + X + 1)  (X^3 - X + 1) (X^3 - X^2 + 1)=\\
		&=& X^8 - X^6 + X^5 + X^4 + X^3 - X^2 + 1
\end{array} 
\] that, in turn, has its own Littlewood multiple.
\end{example}

The last example in this subsection illustrates the ability of Algorithm~\ref{bfs} to work with polynomials $P(X)$ with repeated noncyclotomic roots.

\begin{example}\label{exmr}

The polynomial $p(X) = X^3-X+1$ has a Newman multiple (see Example \ref{psmp}). However, its square $p(X)^2$ does not.

The square $p(X)^2$ divides a Littlewood polynomial $L(X)$ of degree $195$ from Table~\ref{PSPL}, while the cube $p(X)^3$ has no Littlewood multiple at all.

These two facts imply that the Borwein multiple
\[
P(X) = (X^2+X+1)p(X)^2=X^8 + X^7 - X^6 + X^4 + X^3 - X + 1
\] of $p^2(X)$ divides no Newman polynomial, but $P(X)$ has a Littlewood multiple, namely the polynomial $L(X)\Phi_3(X^{196})$, where $\Phi_3(X)=X^2+X+1$.

\begin{table}[h!]
 \caption {Coefficients $l_0,l_1,\dotsc,l_{195}\in\{-1,1\}$ of the Littlewood 
 multiple $L(X)=\sum_{j=0}^{195}l_jX^{195-j}$ of
 $p(X)=(X^3-X+1)^2$.}\label{PSPL}
\begin{tabular}{l}
\toprule
$++-++-+----+++++-+-++-+++++-+-$\\
$++++----+++-+++-++++-+-++++-+-$\\
$+--+-+-+-++--+-+-++++-++---++-$\\
$+-+-+-+++-+-++-++------++-++-+$\\
$-++++-++-+-+---++++----+++-+++$\\
$++-+++++-----+-----+---+++-+--$\\
$-++---++-+--+-+\,-$\\
\bottomrule
 \end{tabular}
 \end{table}

\end{example}

\subsection{Irreducible non-cyclotomic polynomials with unimodular roots}\label{unimod}

In the context of the work of Borwein and Hare \cite{BH}, Stankov \cite{dst} on the spectra of Salem numbers and Salem numbers that are roots of Newman polynomials by Hare and Mossinghoff \cite{HM}, we also investigated the subset $\uu_{\leq 9}^{irr}$ of monic irreducible non-cyclotomic Borwein polynomials of degree at most $9$ with unimodular roots. The set $\uu_{\leq 9}^{irr}$ contains exactly $52$ polynomials. It can be partitioned in to $3$ disjoint subsets
\[
\uu_{\leq 9}^{irr} = \uu_{\leq 9}^1 \cup \uu_{\leq 9}^2 \cup \uu_{\leq 9}^{spor},
\] where:

\begin{itemize}
\item[--] $\uu_{\leq 9}^1$ consists of $28$ minimal polynomials of Salem numbers (\emph{Salem polynomials}) or minimal polynomials of negative Salem numbers ($\al$ is a negative Salem number if $-\al$ is a Salem number). Salem polynomials are given in  Table~\ref{MPSN1};  negative--Salem polynomials can be obtained by substitution $X \to -X$. All $P(X)$ from Table~\ref{MPSN1} belong to the set $ \sm{\ll(\bb)}{\nn(\bb)}$.

\item[--] $\uu_{\leq 9}^2$ consists of $19$ minimal polynomials of complex Salem numbers. $P(X)\in\uu_{\leq 9}^{irr}$ is a complex Salem polynomial if exactly four of its roots, $\{z,  \overline{z}, z^{-1},\overline{z}^{-1}\}$ do not lie on the unit circle. These polynomials $P(X)$ are shown in Table~\ref{MPCSN1}, where  $P(-X)$ are omitted.  All but one (no. 4) polynomials from  Table~\ref{MPCSN1} belong to $\ll(\bb)$. Only polynomials no. 5, 6 and 8 of Table~\ref{MPCSN1} belong to $\nn(\bb)$.

\item[--] $\uu_{\leq 9}^{spor}$ contains remaining $5$ `sporadic' cases from $\uu_{\leq 9}^{irr}$; these polynomials are listed in Table~\ref{SBPUNR}; $P(-X)$ are omitted. Polynomial no.1 has $2$ unimodular roots; no. 2 and 3 has $4$ unimodular roots each. All polynomials from Table~\ref{SBPUNR} belong to $\sm{\ll(\bb)}{\nn(\bb)}$.

\end{itemize}

\begin{table}[h]
 \caption { Salem polynomials from $\uu_{\leq 9}^1$.}\label{MPSN1}
\begin{tabular}{lc}
\toprule
$P(X) \in \bb_{\leq 9}$ &  $P(-X)\in\nn(\bb)$\\
\midrule
$ X^{4} - X^{3} - X^{2} - X + 1 $ & \emph{no} \\
$ X^{6} - X^{5} - X^{4} - X^{3} - X^{2} - X + 1 $ & \emph{no} \\
$ X^{6} - X^{5} - X^{4} - X^{2} - X + 1 $ & \emph{no} \\
$ X^{6} - X^{5} - X^{4} + X^{3} - X^{2} - X + 1 $ & \emph{yes} \\
$ X^{6} - X^{5} - X^{3} - X + 1 $ & \emph{yes} \\
$ X^{6} - X^{4} - X^{3} - X^{2} + 1 $ & \emph{yes}\\
$ X^{8} - X^{7} - X^{6} - X^{5} - X^{3} - X^{2} - X + 1 $ & \emph{no} \\
$ X^{8} - X^{7} - X^{6} - X^{4} - X^{2} - X + 1 $ & \emph{no} \\
$ X^{8} - X^{7} - X^{6} - X^{2} - X + 1 $ &  \emph{no} \\
$ X^{8} - X^{7} - X^{6} + X^{4} - X^{2} - X + 1 $ & \emph{yes} \\
$ X^{8} - X^{7} - X^{5} - X^{4} - X^{3} - X + 1 $ & \emph{no} \\
$ X^{8} - X^{7} - X^{5} + X^{4} - X^{3} - X + 1 $ & \emph{yes} \\
$ X^{8} - X^{6} - X^{5} - X^{3} - X^{2} + 1 $ & \emph{yes} \\
$ X^{8} - X^{5} - X^{4} - X^{3} + 1 $ &  \emph{yes} \\
\bottomrule
 \end{tabular}
 \end{table}

\begin{table}[h]
 \caption {Complex Salem polynomials $\uu_{\leq 9}^2$; $P(-X)$ omitted}\label{MPCSN1}
\begin{tabular}{rlc}
\toprule
\# & $P(X) \in \bb_{\leq 9}$ &   $P(-X)\in\nn(\bb)$\\
\midrule
1 & $ X^{6} - X^{5} + X^{4} + X^{3} + X^{2} - X + 1 $  &  \emph{no} \\
2 & $ X^{8} - X^{7} - X^{6} + X^{5} + X^{4} + X^{3} - X^{2} - X + 1 $  & \emph{yes} \\
3 & $ X^{8} - X^{7} + X^{5} + X^{3} - X + 1 $ &  \emph{yes} \\
4 & $ X^{8} - X^{7} + X^{5} + X^{4} + X^{3} - X + 1 $  &  \emph{no} \\
5 & $ X^{8} - X^{7} + X^{6} - X^{4} + X^{2} - X + 1 $  & \emph{yes} \\
6 & $ X^{8} - X^{7} + X^{6} + X^{4} + X^{2} - X + 1 $  & \emph{yes} \\
7 & $ X^{8} - X^{7} + X^{6} + X^{5} + X^{4} + X^{3} + X^{2} - X + 1 $  &  \emph{no} \\
8 & $ X^{8} + X^{5} + X^{4} + X^{3} + 1 $ &   \emph{yes} \\
9 & $ X^{8} + X^{6} - X^{4} + X^{2} + 1 $  &  \emph{no} \\
10 & $ X^{8} + X^{6} + X^{5} - X^{4} + X^{3} + X^{2} + 1 $  &  \emph{no} \\
\bottomrule
 \end{tabular}
 \end{table}

\bigskip

\begin{table}[h]
 \caption {Sporadic polynomials from $\uu_{\leq 9}^{spor}$; $P(-X)$ omitted}\label{SBPUNR}
\begin{tabular}{lc}
\toprule
$P(X) \in \bb_{\leq 9}$ &  $P(-X)\in\nn(\bb)$\\
\midrule
$ X^{8} - X^{7} + X^{6} - X^{5} - X^{4} - X^{3} + X^{2} - X + 1 $ &  \emph{no} \\
$ X^{8} - X^{7} - X^{6} + X^{5} - X^{4} + X^{3} - X^{2} - X + 1 $ &  \emph{no} \\
$ X^{8} - X^{6} - X^{4} - X^{2} + 1 $ &   \emph{no} \\
\bottomrule
 \end{tabular}
 \end{table}

We end Section \ref{unimod} by demonstrating a few notable examples of $P(X) \in \uu_{\leq 9}^{irr}$.

\begin{example}\label{ex1uni}
Complex Salem polynomials
\[
P(X) = X^{8} - X^{7} + X^{6} - X^{4} + X^{2} - X + 1
\]
and $P(-X)$ belong to $\ll(\bb)\cap\nn(\bb)$. 

In contrast, complex Salem polynomials 
\[
Q(X) = X^{8} - X^{7} + X^{5} + X^{4} + X^{3} - X + 1
\]
and $Q(-X)$ are in $\sm{\bb}{(\ll(\bb)\cup\nn(\bb))}$. Moreover, $Q(X)$ and $Q(-X)$ are two unique polynomials from $\uu_{\leq 9}^{irr}$ with no Littlewood multiple. (The roots of $P(X)$ and $Q(X)$ are depicted  in Figure~\ref{ex1P} and \ref{ex1Q}.)

\end{example}

\begin{example}\label{ex2uni}
Sporadic polynomials
\[
P(X) = X^{8} - X^{7} - X^{6} + X^{5} - X^{4} + X^{3} - X^{2} - X + 1
\]
and $P(-X)$ have 4 unimodular roots and 4 real roots (positive and negative, see Figure~\ref{ex2P}).  

Sporadic polynomials
\[
Q(X) = X^{8} - X^{7} + X^{6} - X^{5} - X^{4} - X^{3} + X^{2} - X + 1
\]
and $Q(-X)$ have exactly 2 unimodular roots each, see Figure~\ref{ex2Q}. It is notable that $\{P(\pm X), Q(\pm X)\}  \subset \sm{\ll(\bb)}{\nn(\bb)}$.
\end{example}

\begin{figure}[h]
	\begin{subfigure}{.45\linewidth}
		\centering
      		\includegraphics[scale=0.5]{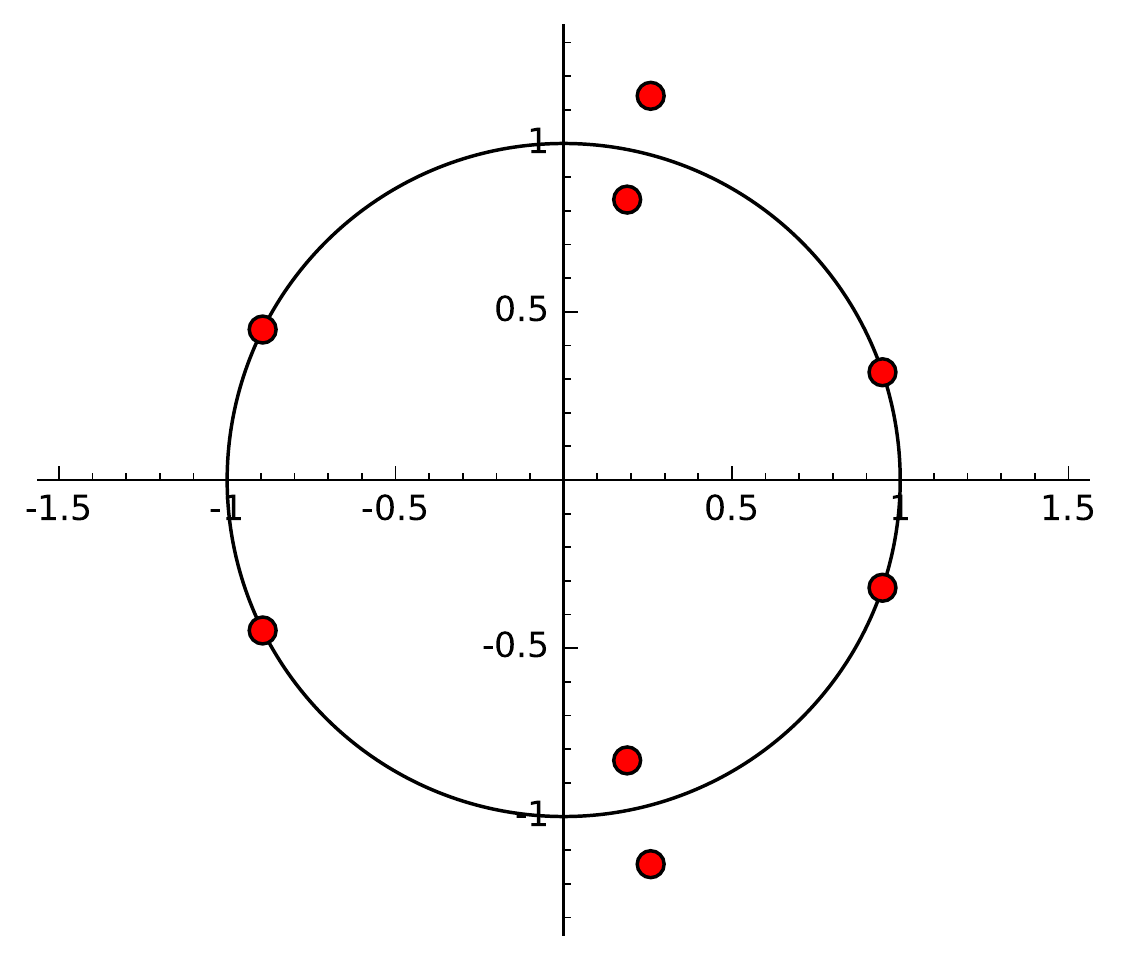}
      		\caption{$X^{8} - X^{7} + X^{6} - X^{4} + X^{2} - X + 1$.}\label{ex1P}
	\end{subfigure}
	\begin{subfigure}{.45\linewidth}
      		\centering
      		\includegraphics[scale=0.5]{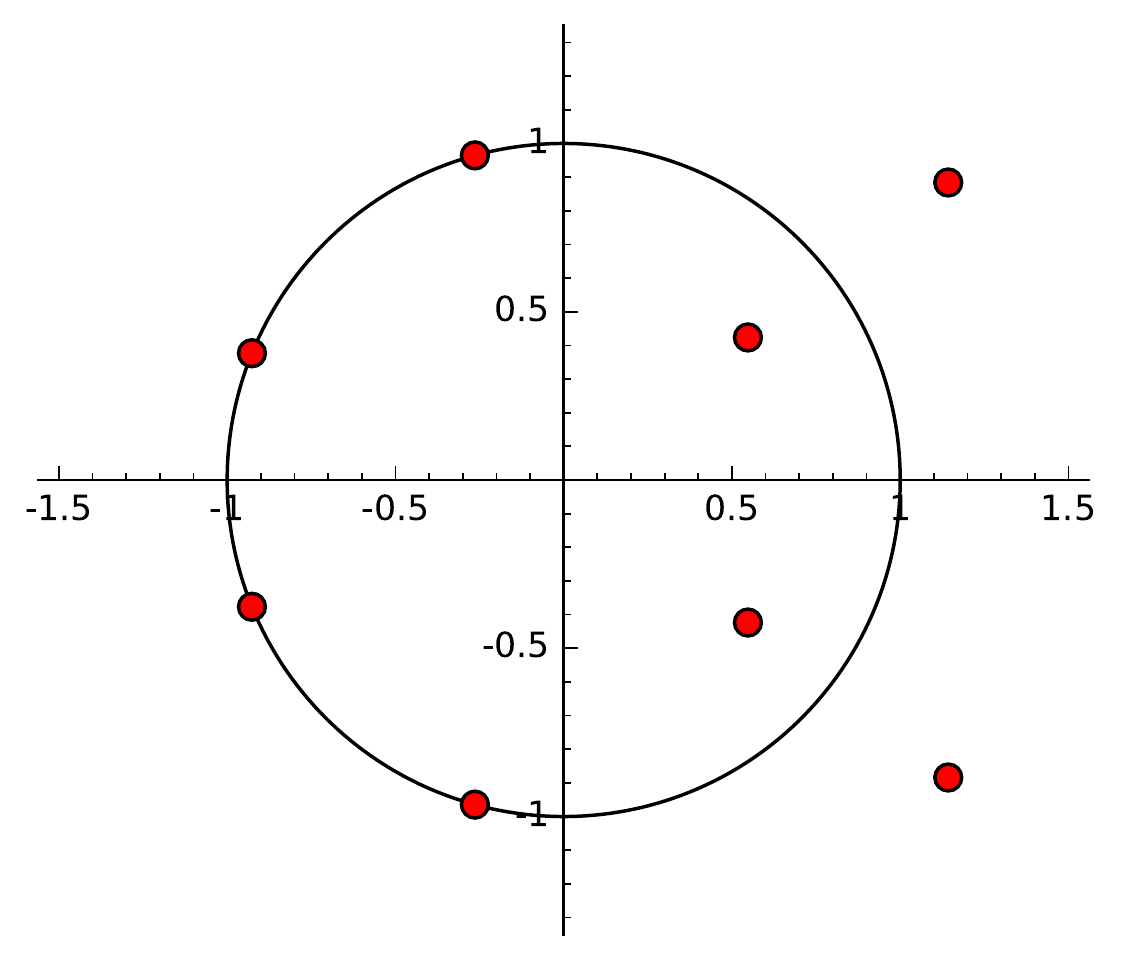}
      		\caption{$X^{8} - X^{7} + X^{5} + X^{4} + X^{3} - X + 1$.}\label{ex1Q}
        \end{subfigure}
	\begin{subfigure}{.45\linewidth}
		\centering
		\includegraphics[scale=0.5]{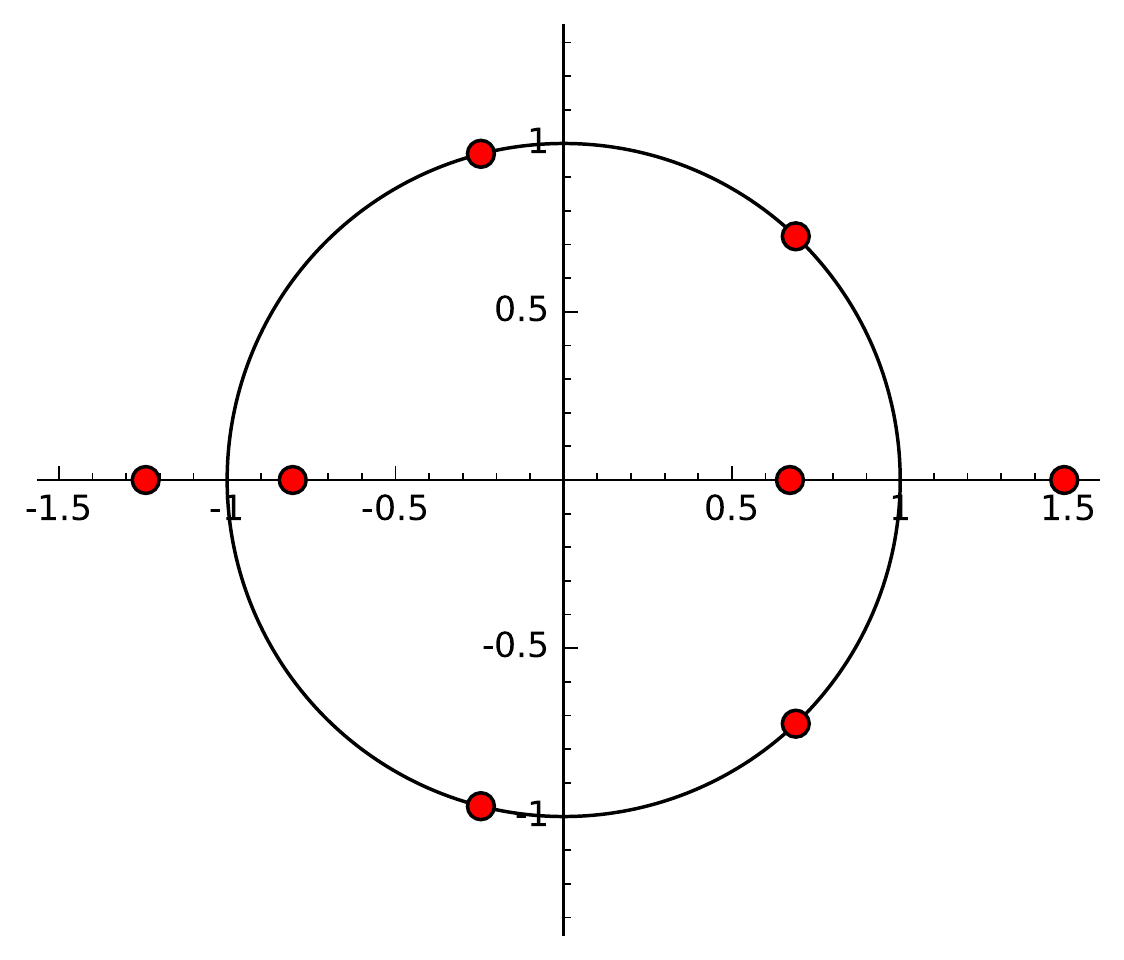}
		\caption{$\protect\begin{array}{l}
					X^{8} - X^{7} - X^{6} + X^{5} - X^{4}+\\
        					+ X^{3} - X^{2} - X + 1.
				\protect\end{array}$}\label{ex2P}
		\label{poly5}
        \end{subfigure}
        \begin{subfigure}{.45\linewidth}
        	\centering
        	\includegraphics[scale=0.5]{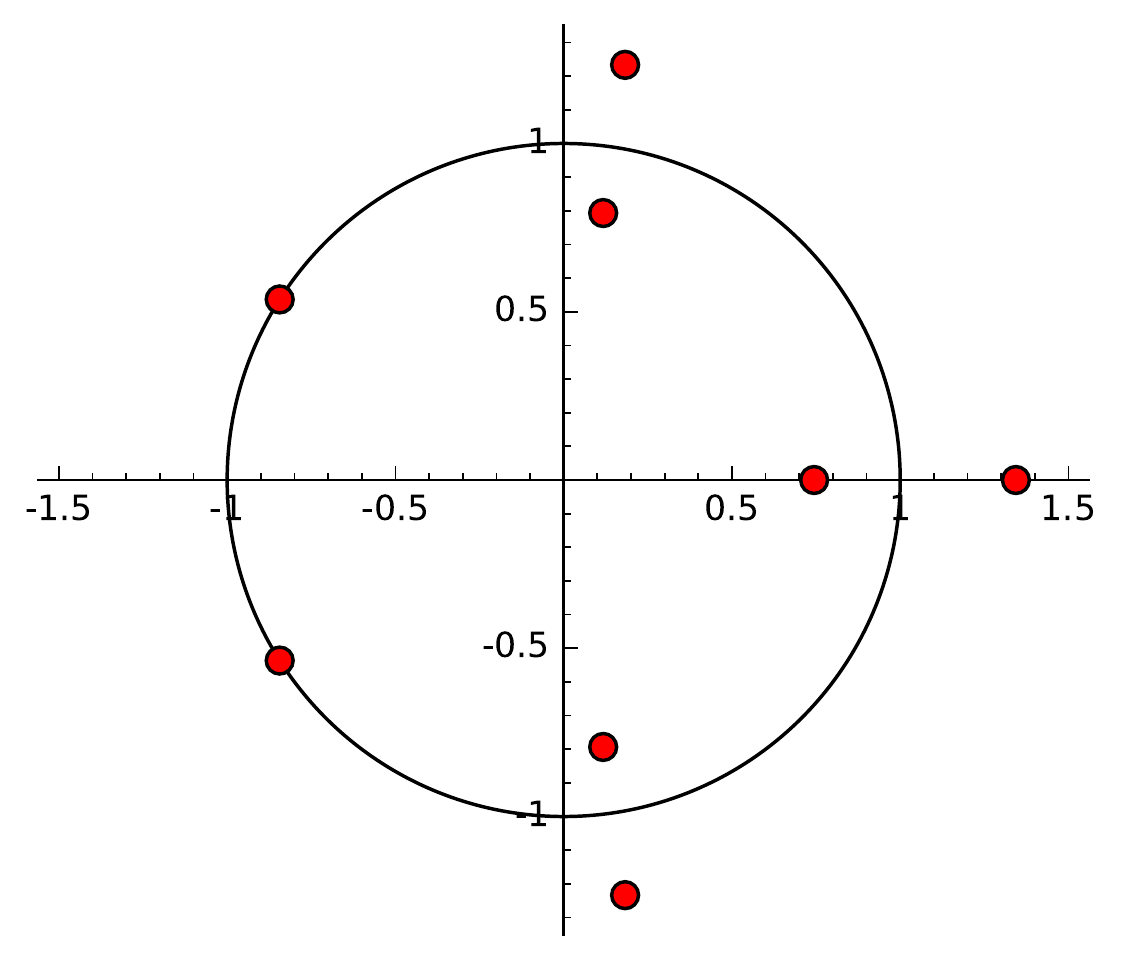}
        	\caption{$\protect\begin{array}{l}
        					X^{8} - X^{7} + X^{6} - X^{5} - X^{4}-\\
						- X^{3} + X^{2} - X + 1.
			       \protect\end{array}$}\label{ex2Q}
		\label{poly4}
	\end{subfigure}
	\caption{Complex roots of $P(X), Q(X)$ from Example \ref{ex1uni} and Example \ref{ex2uni}} 
\end{figure}\label{ex12fig}

\section{Computations}\label{calc}

Assume that $p(X)$ is a nonzero polynomial with integer coefficients and recall that \emph{the content} of $p(X)$ 
is the greatest common divisor of all of its coefficients.  Suppose we have a factorization $p(X) =a \cdot C(X)N(X)$, 
where $a\in\Z$, $C(X),N(X)\in\Z[X]$, the polynomial 
$C(X)$ is a product of cyclotomic polynomials, whereas the polynomial $N(X)$  has no cyclotomic divisors, the 
content of $N(X)$ equals $1$ and the leading coefficient of $N(X)$ is a positive integer. Then $N(X)$ is 
called \emph{the noncyclotomic part of} $p(X)$ and the polynomial $C(X)$ is called the \emph{cyclotomic part of} $p(X)$. 
Note that the noncyclotomic part of a polynomial is uniquely determined.

The set $\bb_{\leq 9}$ is the union of the following disjoint subsets (see Table~\ref{CLa}):
\begin{itemize}
\item[$\cc$]  -- the set of polynomials from $\bb_{\leq 9}$ which are products of cyclotomic polynomials;
\item[$\ff^1$]  --  the set of polynomials from $\bb_{\leq 9}$ whose noncyclotomic part is an irreducible nonconstant polynomial;
\item[$\ff^2$]  --  the set of polynomials from $\bb_{\leq 9}$ whose noncyclotomic part is the product of two distinct monic irreducible nonconstant polynomials;
\item[$\mm$]  --  the set of polynomials from $\bb_{\leq 9}$ whose noncyclotomic part is the square of a monic irreducible nonconstant polynomial.  
\end{itemize}

The same classification of the elements of $\bb_d$ is also valid for degrees $d=10$ and $11$, but no longer holds 
for $\bb_{12}$.

The numbers (computed with SAGE \cite{SAGE}) $\# \bb_d$, $\# \cc_d$, $\# \ff^1_d$, 
$\# \ff^2_d$, $\# \mm_d$, for $d\in\{1,2,\dotsc,9\}$, are given in Table~\ref{CLa}. 
(Recall that $\mathcal{A}_d$ denotes the set of polynomials from $\mathcal{A}$ of degree $d$.)
In particular, $\#\bb_{\leq 9} = 39364$. 

\begin{table}[h]
 \caption {Partition of the set $\bb_{\leq 9}$.}\label{CLa}
\begin{tabular}{rccccc}
\toprule
 $d$ & $\# \bb_d$ & $\# \cc_d$ & $\# \ff^1_d$ & $\# \ff^2_d$ & $\# \mm_d$ \\
\midrule
$ 1 $ & $ 4 $ & $ 4 $ & $ 0 $ & $ 0 $ & $ 0 $ \\
$ 2 $ & $ 12 $ & $ 8 $ & $ 4 $ & $ 0 $ & $ 0 $ \\
$ 3 $ & $ 36 $ & $ 12 $ & $ 24 $ & $ 0 $ & $ 0 $ \\
$ 4 $ & $ 108 $ & $ 20 $ & $ 88 $ & $ 0 $ & $ 0 $ \\
$ 5 $ & $ 324 $ & $ 32 $ & $ 292 $ & $ 0 $ & $ 0 $ \\
$ 6 $ & $ 972 $ & $ 48 $ & $ 892 $ & $ 32 $ & $ 0 $ \\
$ 7 $ & $ 2916 $ & $ 68 $ & $ 2784 $ & $ 64 $ & $ 0 $ \\
$ 8 $ & $ 8748 $ & $ 96 $ & $ 8352 $ & $ 292 $ & $ 8 $ \\
$ 9 $ & $ 26244 $ & $ 136 $ & $ 25228 $ & $ 880 $ & $ 0 $ \\
\bottomrule
 \end{tabular}
 \end{table}

We implemented Algorithm~\ref{al1} in C using library Arb \cite{arb} for arbitrary-precision floating-point 
ball arithmetic and ran it on the SGI Altix 4700 server at
Vilnius University.  We used OpenMP \cite{omp} for an implementation of multiprocessing. 

For every Borwein polynomial $p(X)$ of degree at most 9 we calculated whether it divides some Littlewood polynomial 
as well as whether $p(X)$ divides some Newman polynomial. Moreover, for every Newman polynomial of degree 
at most $11$ we calculated whether it has a Littlewood multiple.  We will briefly explain how these calculations 
were organized. 

First, note that in view of  Proposition~\ref{ocd} a polynomial $P(X)\in\Z[X]$ divides some Littlewood polynomial if and only 
if its noncyclotomic part divides some Littlewood polynomial.  
Similarly, if $P(1)\neq 0$ then $P(X)$ has a Newman multiple if and only 
if its noncyclotomic part has a Newman multiple. (Note that Newman polynomials do not have 
nonnegative real roots.) Therefore when considering the statements $P(X)\in\ll(\bb)$ and $P(X)\in\nn(\bb)$ we can omit 
the cyclotomic part of the polynomial $P(X)$. Also, by Proposition~\ref{ocd}, if $P(X)\in\cc$ then $P(X)$ divides 
some Littlewood polynomial; $P(X)\in\cc$ divides some Newman polynomial if and only if $P(1)\neq 0$.

For each noncyclotomic irreducible factor of polynomials from $\bb_{\leq 9}$ we ran our algorithm and calculated 
whether it has a Littlewood multiple and whether it has a Newman multiple. This allowed us to easily verify the 
statements $P(X)\in\ll(\bb)$ and $P(X)\in\nn(\bb)$ for polynomials $P(X)\in\ff^1$. Further when considering the 
statement $P(X)\in\ll(\bb)$ we omitted those polynomials $P(X)$ from $\ff^2$ and $\mm$ which had a noncyclotomic 
irreducible factor that does not divide any Littlewood polynomial. The procedure for calculating Newman multiples was the same.    
Finally, we ran our algorithm for noncyclotomic parts of the remaining polynomials from $\ff^2$ and $\mm$.
 
 We often used the following two facts to decide that a polynomial has no Newman multiple. First, Newman polynomials 
do not have nonnegative real roots. On the other hand, Odlyzko and Poonen \cite{OP} proved that roots of Newman polynomials 
are contained in the annulus $1/\tau < |z| < \tau$, where $\tau = (1+\sqrt{5})/2\approx 1.61803$ is the golden ratio.

There are exactly $376$ Borwein polynomials of degree at most 9 that have unimodular roots which are not 
roots of unity. For every such polynomial we ran  Algorithm~\ref{al1} and omitted unimodular roots when 
checking the condition \eqref{eq:237con} (see the note after Algorithm~\ref{al1}). We succeeded in deciding 
whether these polynomials belong to $\ll(\bb)$ and $\nn(\bb)$.

We introduced a new variable $0\leq \delta <1$ (note that $B=\max\{|a|\mid a \in \dd\}=1$ in case of  
Littlewood and Newman multiples) to fasten the search for  Littlewood and Newman multiples of polynomials 
in $\ll(\bb_{\leq 9})$ and $\nn(\bb_{\leq 9})$, respectively (see Figure~\ref{fdpp}). 
For a given $\delta$ we changed the condition \eqref{eq:237con} 
as follows:

\begin{align*}
  |R(\al_j)| &\leq \frac{B - \delta}{||\al_j|-1|},\;\,\\[0,1cm]
  |R'(\al_j)| &\leq \frac{1!(B -\delta)}{||\al_j|-1|^2},\\
  &\dotsb\\
  |R^{(e_j-1)}(\al_j)| &\leq
  \frac{(e_j-1)!(B-\delta)}{||\al_j|-1|^{e_j}}.
\end{align*}
This eliminates some of the vertices in the original graph $\G(P, \dd)$. We start  
with the initial value $\delta = 0.95$.  If a Littlewood (or Newman) multiple is found then we are done. 
Otherwise we decrease $\delta$ by $0.05$ and try again. Note that for polynomials in $\bb\!\setminus\!\ll(\bb)$ and 
$\bb\!\setminus\!\nn(\bb)$ the variable $\delta$ always reaches the value $\delta=0$ in order to construct the 
full graph $\G(P, \dd)$. 

\begin{figure}[h]
    \caption{Distribution of noncyclotomic factors $F(X)$ of polynomials from $\bb_{\leq 9}$
    such that $F(X)\in\ll(\Z[X])$.}\label{fdpp}
      \centering
      \includegraphics[scale=0.5]{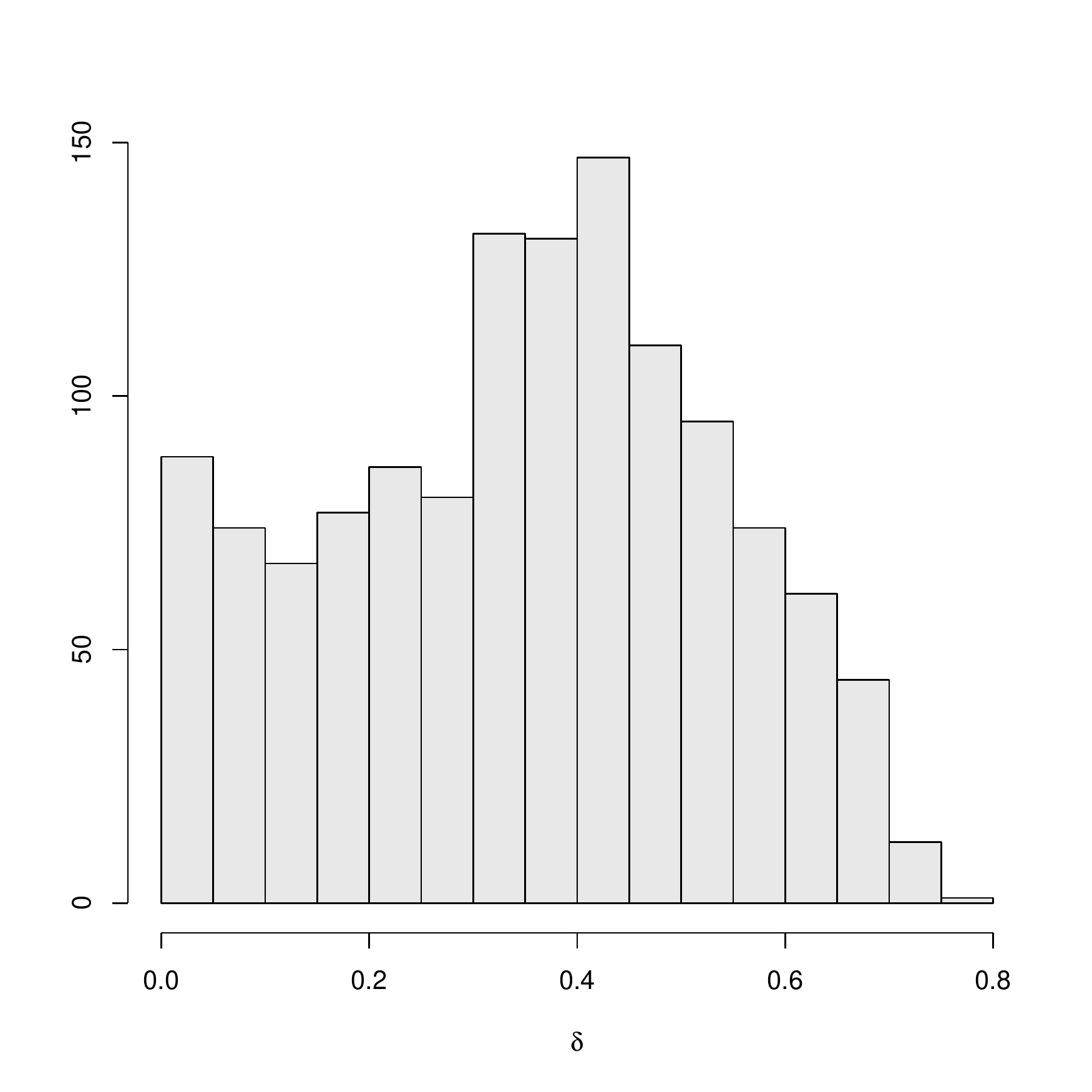}
        \end{figure}

The above mentioned computations took approximately $296$ hours of CPU time. 
The maximum recursion depth reached when searching for Littlewood multiples was $57\,767$, whereas for 
Newman multiples it was $825$.  For instance, it took approximately $119$ minutes of CPU time to run our 
algorithm to decide that the polynomial
\[
X^{9} + X^{8} - X^{7} - X^{5} + X^{3} + X^{2} - 1
\]
has no Littlewood multiple. The graph $\G(P,\dd)$, constructed for this polynomial, contained $1\,428\,848$ vertices. The 
maximal recursion depth reached for this polynomial was $471$. On the other hand, it took $92$ minutes 
of CPU time to find a Littlewood multiple for the polynomial 
\[
X^{9} - X^{8} + X^{7} + X^{6} - X^{5} + X^{4} - X^{3} + X - 1.
\]
The graph $\G(P,\dd)$, constructed for this polynomial, contained $9\,372\,425$ vertices and the maximal recursion depth 
was $43554$.

\subsection{Omitting cyclotomic factors}\label{nocycl}

Given a set $X$ of numbers denote by $-X$ the set $\{ -x\mid x\in X\}$.

\begin{lemma}\label{lm:664}
Let $\Phi_n(X)$ be the $n$-th cyclotomic polynomial. If a positive integer $t$ is 
not divisible by $n$ then $\Phi_n(X)$ divides the polynomial 
$X^{(n-1)t} + X^{(n-2)t} + \dotsb + X^{t} + 1$. 
\end{lemma}

\begin{proof}
By applying formula
\[
X^m - 1 = \prod_{d\mid m}\Phi_d(X),
\]
which is valid for every positive integer $m$, we obtain that $X^{t}-1$ is not 
divisible by $\Phi_n(X)$, because $t$ is not a multiple 
of $n$. Hence $\Phi_n(X)$ and $X^{t}-1$ are coprime, since $\Phi_n(X)$ 
is irreducible. 

Obviously, $X^n-1$ divides $X^{nt}-1$, and therefore $\Phi_n(X)$ divides 
$X^{nt}-1$. On the other hand, $X^{nt}-1$ factors as
\[
X^{nt}-1 = (X^{t}-1)(X^{(n-1)t} + X^{(n-2)t} + \dotsb + X^{t} + 1).
\]
Since $\Phi_n(X)$ is coprime to $X^{t}-1$, we obtain that $\Phi_n(X)$ 
divides the polynomial $X^{(n-1)t} + X^{(n-2)t} + \dotsb + X^{t} + 1$. 
\end{proof}

The following proposition shows that under certain conditions, we can omit its cyclotomic divisors $\Phi_n(X)$, $n>1$
from polynomial $P(X)$ in Problem~\ref{p136}.

\begin{proposition}\label{ocd}

Let $\dd \subset \Z$ be non empty set. Suppose that $\dd$ satisfies at least one of the two conditions: 
\[
0\in\dd  \quad \text{ or } \quad \dd = -\dd.
\] If
$P\in\Z[X]$ divides some nonzero polynomial with coefficients from $\dd$,  then for every positive integer $n>1$, 
 the product $P(X)\Phi_n(X)$, where $\Phi_n(X)$ is the $n$-th cyclotomic 
polynomial, also divides some nonzero polynomial with coefficients from $\dd$.
 
In case $\dd = -\dd$, this is also true for $n=1$: $P(X)(X-1)$ has a non-zero multiple with coefficients from $\dd$.
\end{proposition}

\begin{proof}
Suppose that there exists 
a nonzero polynomial $R\in\Z[X]$ whose all the coefficients are in $\dd$ and which is 
a multiple of $P$. Let $d$ be the degree of $R$.

Assume that $0\in\dd$ and choose an integer $t\geq d+1$, 
which is not divisible by $n$ (e.g., $t = dn+1$). Then all the coefficients of the 
polynomial 
\[
R(X)(X^{(n-1)t} + X^{(n-2)t} + \dotsb + X^{t} + 1)
\]
lie in $\dd$. Moreover, this polynomial is divisible by the product $P\Phi_n$, 
since $P$ divides $R$ and, by Lemma~\ref{lm:664}, 
$X^{(n-1)t} + X^{(n-2)t} + \dotsb + X^{t} + 1$ is 
a multiple of $\Phi_n(X)$. This completes the proof of the proposition in 
 the case when $0\in\dd$.
 
Assume that $\dd = -\dd$. If $n$ divides $d+1$ then, obviously, $X^n-1$ divides 
$X^{d+1}-1$, and therefore $X^{d+1}-1$ is a multiple of $\Phi_n$. Since $\dd = -\dd$, 
all the coefficients of the polynomial $R(X)(X^{d+1}-1)$ lie in $\dd$ and we are done in 
this case. If $d+1$ is not a multiple of $n$ then, by Lemma~\ref{lm:664}, $\Phi_n(X)$ divides 
the polynomial $X^{(n-1)(d+1)} + X^{(n-2)(d+1)} + \dotsb + X^{d+1} + 1$. Finally, 
note that all the coefficients of the 
polynomial 
\[
R(X)(X^{(n-1)t} + X^{(n-2)t} + \dotsb + X^{t} + 1)
\]
lie in $\dd$ and this polynomial is divisible by the product $P\Phi_n$. 

As for the second part of the Proposition note that if  $\dd = -\dd$ then the polynomial $R(X)(X^{d+1}-1)$ 
is divisible by $P(X)(X-1)$ and all of its coefficients belong to $\dd$. 
\end{proof} 
 
\section{The algorithm}\label{sal}

In this section develop an algorithm to answer Question~\ref{p136}.

\begin{lemma}\label{ineq}
Suppose that $z \in \C$ is a 
root of multiplicity $m \geq 1$ of the polynomial
\[
Q(X) = a_d X^d + a_{d-1} X^{d-1} + \dots +a_1 X + a_0 \in \C[X]
\] of degree $d \geq 1$. Let $j\in\{1,\dotsc, d\}$ and 
\[
R(X) = a_d X^j + a_{d-1} X^{j-1} + \dots + a_{ d-j}.
\]
If $\abs{z} \ne 1$, then, for each $k\in\{ 0,1,\dotsc, m-1\} $, the inequality
\begin{equation}\label{eq:156n}
|R^{(k)}(z)| \leq \frac{k!\cdot H(Q)}{||z|-1|^{k+1}}
\end{equation}
holds. Here $R^{(k)}$ denotes the $k$th derivative of the polynomial $R$, 
$R^{(0)}:=R$, and $H(Q)$ stands for the height of  the polynomial $Q$, namely,
\[
H(Q) = \max\{\abs{a_d}, \abs{a_{d-1}}, \dots, \abs{a_1}, \abs{a_0}\}.
\]

\end{lemma}

\begin{proof}
First, assume that $|z| > 1$. Since $z$ is a root of $Q(X)$ of multiplicity $m$, 
there exists a polynomial $T(X)\in\C[X]$ such that 
\[
a_d X^d + a_{d-1} X^{d-1} + \dots +a_1 X + a_0 =  T(X) \cdot (X - z)^m.
\]
One has 
\[
X^{d-j}\big( a_d X^j + a_{d-1} X^{j-1} + \dots + a_{ d-j}\big) + a_{ d-j-1}X^{d-j-1} + \dotsb + 
a_0 = T(X) \cdot (X - z)^m,
\]
and so  
\[
R(X) = a_d X^j + a_{d-1} X^{j-1} + \dots + a_{ d-j}  
= -\frac{a_{ d-j-1}}{X} - \dotsb - \frac{a_0}{X^{d-j}} + 
\frac{T(X) \cdot (X - z)^m}{X^{d-j}}.
\]

Now fix $k\in\{ 0,1,\dotsc, m-1\}$. One can easily see that the $k$th derivative of 
the rational function $T(X) \cdot (X - z)^m/X^{d-j}\in\C(X)$ vanishes at $X = z$. Therefore 
\[
R^{(k)}(z) = \left.\Big(  -\frac{a_{ d-j-1}}{X} - \dotsb - \frac{a_0}{X^{d-j}} 
\Big)^{(k)}\right|_{X=z}
\]
\vspace{0,2cm}
\[
= (-1)^{k+1}\frac{k! \,a_{ d-j-1}}{z^{k+1}} + (-1)^{k+1}\frac{(k+1)!\, 
a_{ d-j-2}}{1!\, z^{k+2}} + \dotsb 
+ (-1)^{k+1}\frac{(d+k-j-1)! \,a_{ 0}}{(d-j-1)!\, z^{d+k-j}}.
\]
From this we obtain
\begin{align}
|R^{(k)}(z)| &\leq H(Q) \Big( 
\frac{k!}{|z|^{k+1}} +\frac{(k+1)!}{1!\, |z|^{k+2}} + 
\dotsb + \frac{(d+k-j-1)!}{(d-j-1)!\, |z|^{d+k-j}}\Big)\nonumber\\[0,3cm]
&\leq H(Q) \Big(\frac{k!}{|z|^{k+1}} +\frac{(k+1)!}{1!\, |z|^{k+2}} + 
\dotsb + \frac{(d+k-j-1)!}{(d-j-1)!\, |z|^{d+k-j}} + \dotsb\Big)\nonumber\\[0,3cm]
&= H(Q) \left.(-1)^k \Big(\frac{1}{X} + \frac{1}{X^2} + \dotsb \Big)^{(k)} \right|_{X=|z|}\nonumber\\[0,3cm]
&= H(Q) \left.(-1)^k \Big(\frac{1}{X-1} \Big)^{(k)} \right|_{X=|z|} = 
\frac{k! H(Q)}{(|z|-1)^{k+1}}.\label{eq:226e} 
\end{align}

Now assume that $|z|<1$. If $k>j = \deg R$ then $R^{(k)}(X)\equiv 0$ and the inequality \eqref{eq:156n} 
obviously holds. Hence assume that $k\leq j$. Then 
\[
R^{(k)}(z) = \left.\Big(   a_d X^j + a_{d-1} X^{j-1} + \dots + a_{ d-j} 
\Big)^{(k)}\right|_{X=z} =
\]
\[
= \frac{j!}{(j-k)!} a_d z^{j-k} + \dotsb + \frac{(k+1)!}{1!}a_{d-j+k+1}z + k! a_{d-j+k},
\]
and therefore 
\begin{align}
|R^{(k)}(z)| &\leq H(Q) \Big( \frac{j!}{(j-k)!} |z|^{j-k} + \dotsb + \frac{(k+1)!}{1!}|z| + k!\Big)\nonumber\\[0,3cm]
&\leq H(Q) \Big(k! + \frac{(k+1)!}{1!}|z| + \dotsb + \frac{j!}{(j-k)!} |z|^{j-k} + \dotsb \Big)\nonumber\\[0,3cm]
&= H(Q) \left. \Big( 1 + X + X^2 + \dotsb \Big)^{(k)} \right|_{X=|z|}\nonumber\\[0,3cm]
&= H(Q) \left. \Big(\frac{1}{1-X}\Big)^{(k)} \right|_{X=|z|} = \frac{k! H(Q)}{(1-|z|)^{k+1}}.\label{eq:242f}
\end{align}
From \eqref{eq:226e} and \eqref{eq:242f} we obtain
\[
|R^{(k)}(z)| \leq \frac{k!\cdot H(Q)}{||z|-1|^{k+1}}.
\]
\end{proof}

Let $P \in \Z[X]$ be a monic polynomial (that is, the leading coefficient of $P$ is equal to $1$). Then one can divide any integer polynomial $Q$ by $P$ in $\Z[X]$: there exist unique integer quotient and remainder polynomials $S$ and $R$, $\deg{R} < \deg{P}$, such that $Q = P \cdot S + R.$ The first key observation: polynomials $S$ and $R$ have integer coefficients, provided that $P$ is monic. The second key observation is as follows. For any complex number $z$ which satisfies $P(z) = 0$, one has $Q(z) = R(z)$. This means that the values of the polynomial $Q$ evaluated at any complex root of the divisor polynomial $P$ coincide with the values of the remainder polynomial $R$ evaluated at same points.

The reduction map $Q \to Q \pmod{P}$ is a homomorphism of rings which maps the ring $\Z[X]$ to the quotient ring $\Z[X]/(P)$. The remainder polynomial $R$ is a representative integer polynomial for the class in $\Z[X] / (P)$ to which $Q$ belongs.

\begin{definition}\label{rem} Let $P(X)$ be a nonconstant polynomial with integer coefficients with no roots on the complex unit circle $|z|=1$. Suppose that the  
factorization of $P$ in $\C[X]$ is 
\[
P(X) = a\cdot (X - \al_1)^{e_1}(X - \al_2)^{e_2}\dotsb (X - \al_s)^{e_s},
\]
where $\al_1, \al_2, \dotsc, \al_s$ are distinct complex numbers and $e_j\geq 1$ for 
$j=1,2,\dotsc, s$. Let $B \in \R$ be arbitrary positive number. Define $\Rem(P, B)$ to 
be the set of all polynomials $R \in \Z[X]$, $\deg R < \deg P$, which, for each 
$j\in\{ 1,2,\dotsc, s \}$, satisfy the  
inequalities
\begin{eqnarray}\label{eq:237con}
|R(\al_j)| \leq \frac{B}{||\al_j|-1|},\;\,\nonumber\\[0,2cm]
|R'(\al_j)| \leq \frac{1!B}{||\al_j|-1|^2},\\
\dotsb\qquad\qquad\;\;\nonumber\\
|R^{(e_j-1)}(\al_j)| \leq \frac{(e_j-1)!B}{||\al_j|-1|^{e_j}}.\nonumber
\end{eqnarray}
Here 
$R^{(j)}$ denotes the $j$th derivative of the polynomial $R$, and $R^{(0)}:=R$.
\end{definition}

The next lemma is of great importance.

\begin{lemma}\label{fin}
 Let $P \in \Z[X]$ and $B \in \R$ be as in Definition~\ref{rem}.  
Then $\Rem(P, B)$ is a finite set.
\end{lemma}

\begin{proof}
Let $P(X)$ be a polynomial of degree $d\geq 1$ with integer coefficients, whose 
factorization in $\C[X]$ is 
\[
P(X) = a\cdot (X - \al_1)^{e_1}(X - \al_2)^{e_2}\dotsb (X - \al_s)^{e_s},
\]
where $\al_1, \al_2, \dotsc, \al_s$ are distinct complex numbers and $e_j\geq 1$ for 
$j=1,2,\dotsc, s$. 

Write
\[
R(X) = r_{d-1} X^{d-1} + \dots + r_1 X + r_0,
\] 
where $r_j,$ $0 \leq j \leq d-1$ are unknown integers. Consider the system of linear equations in variables $r_j$:
\begin{equation}\label{eq:280system}
\begin{split}
r_0+  r_1 \al_j + \phantom{2}r_2 \al_j^2 +\phantom{222222}\dots\phantom{222222}   + \phantom{(d-1)}r_{d-1} \al_j^{d-1}\, &= R(\al_j), \\
         r_1\phantom{\al_j}  +  2r_2 \al_j + \phantom{222222}\dots\phantom{222222} + (d-1)r_{d-1} \al_j^{d-2}\, &= R'(\al_j),\\
                \dotsb \quad \dotsb \quad\quad \dotsb \qquad\qquad\quad\;\,\, &\\
 (e_j-1)!r_{e_j-1} + \dots +\,\,\, {\scriptstyle \frac{(d-1)!}{(d-e_j)!}}r_{d-1}\al_j^{d-e_j} &= R^{(e_j-1)}(\al_j)
\end{split}
\end{equation}
for $j=1, 2, \dots, s$. Write it in the matrix form $A\mathbf{x} = \mathbf{y}$, where
\[
\quad \mathbf{x} = \begin{pmatrix}
r_0\\
r_1\\
\dots\\
r_{d-1}
\end{pmatrix},
\quad \mathbf{y} = \begin{pmatrix}
R(\al_1)\\
R'(\al_1)\\
\dots\\
R^{(e_s-1)}(\al_s)\\
\end{pmatrix}
\]
and the system matrix $A$ is the confluent Vandermond matrix  
which consists of row-blocks ($j =1,2,\dotsc,s$)
\[
\begin{matrix}
1 & \al_j & \al_j^2 & \dotsb & & & \phantom{(d-1)(d-2)}\al_j^{d-1} \\
0 & 1 & 2\al_j & \dotsb & & & \phantom{(d-2)}(d-1)\al_j^{d-2} \\
0 & 0 & 2 & \dotsb & & & (d-1)(d-2)\al_j^{d-3} \\
 &   &  \dotsb     & \dotsb & &  & \dotsb  \\
  0 & 0  &  \dotsb         & (e_j-1)! & e_j!\al_j & \dotsb &  \phantom{(d-2)1}{ \frac{(d-1)!}{(d-e_j)!}}\al_j^{d-e_j}
\end{matrix}
\]
(each row in this block, except for the first one, is the derivative in $\al_j$ of the previous row).
Denote by $D(\al_1^{e_1}\al_2^{e_2}\dotsc \al_s^{e_s})$ the determinant of the confluent Vandermond matrix $A$. 
It is well-known (see, for instance, \cite[Chapter VI]{ait}, \cite[Chapter 6]{horn}, \cite{kal} and 
\cite{mer}) that 
\[
D(\al_1^{e_1}\al_2^{e_2}\dotsc \al_s^{e_s}) = \prod_{i<j} (\al_j - \al_i)^{e_i e_j} \prod_{k=1}^s (e_k - 1)!!,
\]
where $n!!$ stands for the product $n!(n-1)! \dotsb 2! \,1!$. In particular, 
\[
\det(A) = D(\al_1^{e_1}\al_2^{e_2}\dotsc \al_s^{e_s})\neq 0,
\] 
since $\al_1, \al_2, \dotsc, \al_s$ are distinct complex numbers. So the inverse matrix 
$A^{-1}$ exists and $\mathbf{x} = A^{-1} \mathbf{y}.$  By Cramer's formula,
\[
r_j = \frac{1}{\det(A)} \left( R(\al_1) A_{1j+1}+ \dots + 
R^{(e_1-1)}(\al_1) A_{e_1j+1} + \dots   + R^{(e_s-1)}(\al_s) A_{d\, j+1}\right),
\] 
for $j=0,1,\dotsc, d-1$, where $A_{kl}$, $1 \leq k, l \leq d$ are the cofactors of the matrix $A$. 

Now, let $B \in \R$ be arbitrary positive number and assume that $R\in\Rem(P, B)$. 
Then in view of \eqref{eq:237con} we have
\[
|r_j| \leq \frac{B}{|\det(A)|} \left( \frac{|A_{1j+1}|}{||\al_1|-1|}+ \dots + 
\frac{(e_1-1)!|A_{e_1j+1}|}{||\al_1|-1|^{e_1}} + \dots   + \frac{(e_s-1)!|A_{d\, j+1}|}{||\al_s|-1|^{e_s}}\right),
\]
for $j=0,1,\dotsc, d-1$. Therefore the number of solutions $\mathbf{x}\in\Z^{d}$ to 
 \eqref{eq:280system} is finite, i.e., the set $\Rem(P, B)$ is finite. 
\end{proof}

We now define a certain graph which is associated to the set of remainder polynomials, which are bounded at the roots of the polynomial $P$ and the digit set $\dd$.

\begin{definition}
Let $\G=\G(P, \dd)$ be a directed graph whose vertices represent all the distinct polynomials 
$R \in \Rem(P, B)\cup \dd$, where $B=\max\{|b|: b\in \dd\}$.  
We connect the vertices which represent two remainder polynomials $R_i$ and $R_j$, by an edge which points from $R_i$ to $R_j$, if
\[
R_j \equiv X \cdot R_i + b \pmod{P}
\] in $\Z[X]/(P)$ for some digit $b \in \dd$.
\end{definition}
Here is the main theorem of this section.
\begin{theorem}\label{main}
Let $P \in \Z[X]$ be a monic polynomial with no roots on the complex unit circle $|z|=1$. Then $P$ divides an integer polynomial
\[
Q(X) = a_n X^n + a_{n-1} X^{n-1} + \dots  + a_1 X + a_0 \in \C[X]
\] with all the coefficients $a_j \in \dd$ and the leading coefficient $a_n \in \dd$, if and only if the graph $\G = \G(P, \dd)$ contains a path which starts at the remainder polynomial $R(X) = a_n$ and ends at $R(X) = 0$. The length of the path is $n$, where $n$ is the degree of $Q$.
\end{theorem}

\begin{proof}
Let us first prove the necessity. Assume that $P$ divides $Q$, that is, $Q(X) \equiv 0 \pmod{P}$. Define the polynomials
\begin{align*}
Q_0(X) &= a_n,\\
Q_1(X) &= a_n X + a_{n-1},\\
Q_2(X) &= a_n X^2 + a_{n-1}X + a_{n-2},\\
            &\dots& \\
Q_n(X) &=a_n X^n + a_{n-1}X^{n-1} + \dots + a_1 X + a_0.
\end{align*} 
Let $R_j$ be the remainder of $Q_j$ modulo $P$.   
Suppose that the factorization of $P$ in $\C[X]$ is 
\[
P(X) = a\cdot (X - \al_1)^{e_1}(X - \al_2)^{e_2}\dotsb (X - \al_s)^{e_s},
\]
where $\al_1, \al_2, \dotsc, \al_s$ are distinct complex numbers and $e_j\geq 1$ for 
$j=1,2,\dotsc, s$. By Lemma \ref{ineq}, each polynomial $Q_i$, $i=1,2,\dots, n$, satisfies 
the inequalities 
\begin{eqnarray}\label{eq:548cons}
|Q_i(\al_j)| \leq \frac{H(Q)}{||\al_j|-1|},\;\,\nonumber\qquad\\
|Q_i'(\al_j)| \leq \frac{1!H(Q)}{||\al_j|-1|^2},\qquad\\
\dotsb\qquad\qquad\;\;\qquad\nonumber\,\\
|Q_i^{(e_j-1)}(\al_j)| \leq \frac{(e_j-1)!H(Q)}{||\al_j|-1|^{e_j}},\nonumber\,
\end{eqnarray}
for $j=1,2,\dotsc, s$. Moreover, for each $j\in\{1,2,\dotsc, s\}$ and each $i\in\{ 1,2,\dotsc, n\}$ 
\[
R_i^{(k)}(\al_j) = Q_i^{(k)}(\al_j),\;\;0\leq k \leq e_j-1,
\]
since $R_i \equiv Q_i \pmod{P}$. Therefore, in view of \eqref{eq:548cons},  
$R_j$ all belong to the set $\Rem(P, B)$, where $B = H(Q)$.   
Reducing the equality \[Q_{j} = X \cdot Q_{j-1} + a_{n-j}\] with 
$a_{n-j} \in \dd$ modulo $P$ in $\Z[X]/(P)$ yields 
\[
R_{j} \equiv X \cdot R_{j-1} + a_{n-j} \pmod{P}.
\] 
Hence, there exists an edge in the graph $\G$ which connects 
$R_{j-1}$ to $R_j$. Since $Q_n(X) = Q(X) \equiv 0 \pmod{P}$, 
one has $R_n = 0$. Consequently, there exists a path in $\G$ which 
joins $R_0 = a_n$ to the reminder polynomial $R_n = 0$.  

Conversely, assume that there exists a path of length $n$ which connects the $n+1$ vertices $R_0, R_1$, $\dots$, $R_n$ with $R_0 = a_n$ and $R_n = 0$. By the definition of the graph $\G$, there exist coefficients $a_j \in \dd$, $j = 1, \dots, n$, such that
\[R_j \equiv X \cdot R_{j-1} + a_{n-j} \pmod{P}.\] Recursively define the polynomials $Q_0 := R_0 = a_n$, $Q_j := X \cdot Q_{j-1} + a_{n-j}$ for $j=1$, $2$, $\dots$, $n$. By the definition, $Q_j \equiv R_j \pmod{P}$. Then the polynomial
\[
Q(X) := Q_n(X) = a_n X^n + a_{n-1}X^{n-1} + \dotsb + a_1 X+ a_0
\] has all the coefficients $a_j \in \dd$, and $Q_n(X)$ is divisible by $P$ in $\Z[X]$.
\end{proof}

According to Lemma \ref{fin}, the graph $\G = \G(P, \dd)$ is finite. Thus, the polynomial $Q$ with the coefficients in the set $\dd$ may be found by running any path finding algorithm on $\G$. For performance reasons, depth-first search was used.

\newpage

\hrule
\begin{algorithm}\label{al1}{Determines whether $P \in \Z[X]$ has a multiple $Q \in \dd[X]$ with the leading coefficient $a \in \dd$}.
\begin{tabbing}\label{bfs}
Input: \ \ \ \  \= a monic polynomial $P \in \Z[X]$,\\
		    \> the digit set $\dd \subset \Z$,\\
		    \>  the leading coefficient $a \in \dd$, $a \ne 0$.\\
Output:         \> a polynomial $Q \in \dd[X]$ or $\varnothing$, if such $Q$ does not exist\\ 
Variables: \>the set $\V$ of visited vertices of the directed graph $\G=\G(P, \dd)$,\\
                \>the set $\E$ of edges that join vertices of $\V$,\\
		\>found - boolean variable indicating if the search is finished.\\
Method:    \>Depth-first search using Theorem \ref{main}.
\end{tabbing}
\end{algorithm}
\hrule
\begin{tabbing}
 Step \=$0$: set $\V = \varnothing$, $\E = \varnothing$\\
 Step \>$1$: add the polynomial $R=a$ into $\V$\\
 Step \>$2$: set \emph{found} := False\\
 Step \>$3$: call \textbf{do\_search}($a$, \emph{found})\\
Step \>$4$: if $found$ then print $a$\\
	\>  \ \ \ \ else print $\varnothing$\\
	\> \ \ \ end if \\
Step \>$5$:  stop.\\ \\

 procedure \textbf{do\_search}(local var $R \in \Z[X]$, var \emph{found}):\\
	\> local var $S \in \Z[X]$\\
 	\> if $R = 0$ then\\
	\> \ \ \ \ \ \ \= set \emph{found} := True\\
	\> else  \\
	\>             \> for each $d \in \dd$ do\\
	\>	 	\> \ \ \ \ \ \=compute $S := X \cdot R + d \pmod{P}$.\\
        \>             \>            \>if $S \notin \V$ and  $S\in\Rem(P, B)$, where $B:=\max\{|d|: d\in \dd\}$  then\\
        \>             \>            \>\ \ \ \ \ \= add $S$ to $\V$\\
        \>             \>            \>           \>add $d$ as an edge from $R$ to $S$ to $\E$\\
	\>             \>            \>           \> call  \textbf{do\_search}($S$, \emph{found})\\
        \>             \>            \> end if\\
	\>              \>           \>if $found$ then\\
	\>              \>           \>           \> print digit $d$\\
	\>              \>           \>           \> break loop\\
        \>             \>            \> end if\\
        \>             \>end do\\
        \>end if\\
end proc
\end{tabbing}
\hrule
\vspace{0,5cm}
\textbf{Note.} If a polynomial $P(X)\in\Z[X]$ has unimodular roots we can exclude them from \eqref{eq:237con} 
and try to build the graph $\G(P, \dd)$. If the resulting graph is finite then we can still answer 
Question~\ref{p136} for such polynomials.

\end{document}